\documentclass[a4paper,11pt,centertags]{amsart}  
\usepackage{amssymb}  
\newtheorem{thm}{Theorem}
\newtheorem{lemma}{Lemma}

\newtheorem{cor}{Corollary}
\newtheorem{prop}{Proposition}

\theoremstyle{definition}

\theoremstyle{remark}
\newtheorem*{rem*}{Remark}
\newtheorem{rem}{Remark}

\newcommand{\mr}{{\mathbb R}}

\newcommand{\mn}{{\mathbb N}}

\newcommand{\mc}{{\mathbb C}}
\newcommand{\md}{{\mathbb D}}

\newcommand{\mt}{{\mathbb T}}

\renewcommand{\rho}{\varrho}
\newcommand{\eps}{\varepsilon}

\renewcommand{\Im}{\operatorname{Im}}
\renewcommand{\Re}{\operatorname{Re}}

\newcommand{\supp}{\operatorname{supp}}
\newcommand{\dist}{\operatorname{dist}}

\newcommand{\ran}{\operatorname{Ran}}
\newcommand{\hil}{\mathcal{H}}
\newcommand{\bdd}{\mathcal{B}}

\begin{document}

\title[Variation of discrete spectra for n.s.a perturbations]{Variation of discrete spectra for non-selfadjoint perturbations of selfadjoint operators}

\begin{abstract}
Let $B=A+K$ where $A$ is a bounded selfadjoint operator and $K$ is an element of the von Neumann-Schatten ideal $\mathcal{S}_p$ with $p>1$. Let $\{\lambda_n\}$ denote an enumeration of the discrete spectrum of $B$. We show that $\sum_n \dist(\lambda_n, \sigma(A))^p$ is bounded from above by a constant multiple of $\|K\|_p^p$. We also derive a unitary analog of this estimate and apply it to obtain new estimates on zero-sets of Cauchy transforms.

\end{abstract}

\author[M. Hansmann]{Marcel Hansmann}
\address{Faculty of Mathematics\\
Chemnitz University of Technology\\
Chemnitz\\
Germany.}
\email{marcel.hansmann@mathematik.tu-chemnitz.de}


 \maketitle

\section{Introduction}

If a selfadjoint operator $A$ on a separable Hilbert space $\hil$ is perturbed by a non-selfadjoint compact perturbation $K$, then the essential spectra of $B=A+K$ and $A$ coincide. However, the spectrum of $B$ can contain an additional countable set of isolated complex eigenvalues of finite algebraic multiplicity. These \textit{discrete} eigenvalues and their variation with respect to the spectrum of $A$ are the topic of this article. 

The following estimate is one of our main results: If $K=B-A$ is an element of the von Neumann-Schatten ideal $\mathcal{S}_p(\hil)$ for some $p > 1$,  then there exists a constant $C_p$, independent of $A$ and $B$, such that
\begin{equation}
  \label{eq:1}
 \sum_{\lambda \in \sigma_d(B)} \dist(\lambda,\sigma(A))^p \leq C_p \|B-A\|_{p}^p.  
\end{equation}
Here $\sigma(A)$ and $\sigma_d(B)$ denote the spectrum and the discrete spectrum (i.e. the set of all discrete eigenvalues) of $A$ and $B$, respectively, and each eigenvalue is counted according to its algebraic multiplicity. We recall that $\mathcal{S}_p(\hil)$ consists of all compact operators $K$ on $\hil$ whose singular values $s_n(K)$ are $p$-summable and that $\|K\|_p^p=\sum_{n} s_n(K)^p$. The constant $C_p$ in (\ref{eq:1}) tends to infinity for $p \to 1$ and $p \to\infty$. Moreover, for $p=2$ we obtain $C_2=2$. The example
$$ A= \left( 
  \begin{array}{cc}
    0 & 1 \\
    1 & 0
  \end{array} \right), \quad B= \left( 
  \begin{array}{cc}
    0 & 1 \\
    0 & 0  
  \end{array}\right)$$
shows that this value is sharp.

To put estimate (\ref{eq:1}) into perspective, let us take a look at some previous results of this type: If $A$ \emph{and} $B$ are selfadjoint then (\ref{eq:1}) is true for $p \geq 1$ and with  $C_p=1$, as it has been shown by Kato \cite{MR900507}. Bhatia and Elsner \cite{MR1314392} showed the validity of  (\ref{eq:1}) for $p \geq 1$ if $A$ is selfadjoint and $B$ is normal. Relaxing the selfadjointness assumption on $A$,  Bhatia and Davis \cite{MR1702207} proved the validity of (\ref{eq:1}), for $p \geq 1$ and with $C_p=1$, in case that $A,B$ and $B-A$ are normal operators.  Bhatia and Davis' result remains true if $A$ and $B$ (but not necessarily $B-A$) are normal, but only if $p\geq 2$, see Bouldin \cite{MR577759}.  To be precise, most of the above authors consider an estimate which is slightly stronger than (\ref{eq:1}). Namely, given the stated restrictions on $A$ and $B$ they show that
there exist extended enumerations $\{\alpha_j\}$ and $\{\beta_j\}$ of $\sigma_d(A)$ and $\sigma_d(B)$, respectively, such that
$$ \sum_{j} |\beta_j-\alpha_j|^p \leq C_p\|B-A\|_p^p,$$  
where an extended enumeration of $\sigma_d(.)$ is a sequence which contains all discrete eigenvalues and which in addition may contain boundary points of the essential spectrum.

The case of most interest to us, where $A$ is selfadjoint and $B$ is arbitrary, has been studied in the finite-dimensional case by Kahan \cite{MR0369394}, whose work contains a proof of (\ref{eq:1}) for $p=2$,  and, more recently, by Gil' \cite{MR2777268}. Indeed, while Gil' considered estimates on the real parts of the eigenvalues of $B$ only, the validity of (\ref{eq:1}) in the finite-dimensional case can easily be derived from his results and below we will adapt his main idea to prove the estimate in the general case.

One might ask whether (\ref{eq:1}) remains true (with $B$ arbitrary) when the assumptions on $A$ are relaxed. Here in general the answer will be no: For instance, see Remark 2.5 in \cite{Hansmann10}, one can construct a normal (or even unitary) matrix $A \in \mc^{n \times n}$ and a corresponding $B \in \mc^{n \times n}$ with $\|B-A\|_p=1$ such that 
$$ \sum_{\lambda \in \sigma_d(B)} \dist(\lambda,\sigma(A))^p = n.$$
However, even for general $A$ \emph{and} $B$  one can at least show that 
  \begin{equation*}
 \sum_{\lambda \in \sigma_d(B)} \dist(\lambda,\operatorname{Num}(A))^p \leq \|B-A\|_{p}^p,  \qquad p \geq 1,
  \end{equation*}
where $\operatorname{Num}(A)$ denotes the numerical range of $A$, see \cite{Hansmann10}. Since the closure of the numerical range of a selfadjoint operator coincides with the convex hull of its spectrum, this estimate implies that for $A$ selfadjoint with $\sigma(A)=[a,b]$ and $B$ arbitrary 
\begin{equation}\label{eq:hans}
  \sum_{\lambda \in \sigma_d(B)} \dist(\lambda,[a,b])^p \leq \|B-A\|_{p}^p,  \qquad p \geq 1.
\end{equation}
Note that in (\ref{eq:1}) we made no assumptions at all on the structure of the spectrum of the selfadjoint operator $A$. The price we pay for this generality, as compared to estimate (\ref{eq:hans}), is the multiplicative constant $C_p$ occurring in (\ref{eq:1}). In particular, since $C_p \to \infty$ for $p\to1$, we can show the validity of (\ref{eq:1})  only for $p>1$. Whether this exclusion of the case  $p=1$ is really necessary, or whether it is just an artefact of our method of proof, remains an open question.

It is interesting to compare our estimate with another recent result, by Golinskii and Kupin \cite{golinskii11}. Using Blaschke-type estimates for holomorphic  functions on finitely-connected domains they showed (among other things) that if  $A$ is selfadjoint with 
$$\sigma(A)= [a_1,a_2] \cup \ldots \cup [a_{2n-1},a_{2n}], \qquad a_1<a_2<\ldots < a_{2n},$$ and $B$ is arbitrary, then for every $\eps \in (0,1)$ there exists $C=C(\eps,p,\sigma(A))$ such that 
\begin{equation}
  \label{eq:gol}
\sum_{\lambda \in \sigma_d(B)} \frac{\dist(\lambda,\sigma(A))^{p+1+\eps}}{\dist(\lambda,\{a_1,\ldots,a_{2n}\})(1+|\lambda|)} \leq C \|B-A\|_{p}^p, \qquad p \geq 1.
\end{equation}
Since for $\eps \in (0,1)$ we can find $C(\eps,A)$ such that for all $\lambda \in \mc \setminus \sigma(A)$
$$ \frac{\dist(\lambda,\sigma(A))^{p+1+\eps}}{\dist(\lambda,\{a_1,\ldots,a_{2n}\})(1+|\lambda|)} \leq C(\eps,A) \dist(\lambda, \sigma(A))^p,$$
we see that, at least in case $p>1$ and ignoring the constants,  estimate (\ref{eq:1}) is stronger than (\ref{eq:gol}). 
In addition, we note again that, in contrast to (\ref{eq:gol}), estimate (\ref{eq:1}) is valid without any further restrictions on $\sigma(A)$. We should stress that for more specific operators the estimates on holomorphic functions derived in \cite{golinskii11} might lead to better estimates than can be obtained from (\ref{eq:1}), see \cite{MR2773062} for a related discussion. 

Considering applications of our estimate, we note that just as (\ref{eq:hans}) has been used to derive Lieb-Thirring type inequalities for Schr\"odinger operators $-\Delta + V$, with a complex-valued potential $V$ (see \cite{Hansmann10}),  
estimate (\ref{eq:1}) can be used to obtain such inequalities when the Laplacian $-\Delta$ is replaced by a more general selfadjoint operator. Since at the moment of writing of this introduction a preliminary version of this article has been available for some time, we can refer the reader to the work of Golinskii and  Kupin \cite{GK2013}, who used (\ref{eq:1}) to study non-selfadjoint perturbations of a selfadjoint finite band Schr\"odinger operator, and to the work of Sambou \cite{Sambou}, who used (\ref{eq:1}) in the study of non-selfadjoint perturbations of magnetic Schr\"odinger operators.

In this article, we present yet another but less immediate application of (\ref{eq:1}). Namely, we will derive a unitary analog of this estimate and use it to study the distribution of zeros of certain holomorphic functions on the unit disk $\md$. We will show that for every \textit{Cauchy transform} of a finite, complex Borel measure $\mu$ on the unit circle $\mt$, i.e. for every function $h$ of the form
 \[ h(w) = \int_\mt \frac{\mu(d\zeta)}{1-\overline{\zeta}w}, \qquad w \in \md,\]
we have
 \[ \sum_{h(w)=0, w \in \md} \dist(w, \supp(\mu))^p < \infty\] 
for every $p>1$. This condition can be regarded as a softer, stronger version of the well-known Blaschke condition $$\sum_{h(w)=0, w \in \md} (1-|w|) < \infty,$$ see \cite{Cauchy}.

\section{The main result}

Let $\bdd(\hil)$ and $\mathcal{S}_\infty(\hil)$ denote the classes of bounded and compact operators on $\hil$, respectively. We define the real and imaginary part of $T \in \bdd(\hil)$ as
$$ \Re(T) = (T+T^*)/2, \quad \Im(T)= (T-T^*)/(2i).$$
Note that $\Re(T)$ and $\Im(T)$ are selfadjoint and that $T \in \mathcal{S}_p(\hil)$ if and only if both $\Re(T) \in \mathcal{S}_p(\hil)$ and $\Im(T) \in \mathcal{S}_p(\hil)$.

One of the most important ingredients in Gil's proof of (\ref{eq:1}) in the finite-dimensional case \cite{MR2777268} is the following classical result of Macaev \cite{MR0136997} (see also \cite{MR0264447}, Section III.6). It is concerned with the real and imaginary parts of abstract Volterra operators. 
\begin{prop}\label{prop:ma}
Let $T \in \mathcal{S}_\infty(\hil)$ with $\sigma(T)=\{0\}$. If $\Im(T) \in \mathcal{S}_p(\hil)$ for some $p>1$, then $\Re(T) \in \mathcal{S}_p(\hil)$ and there exists a constant $b_p \geq 1$ such that 
\begin{equation}\label{eq:bp}
 \|\Re(T)\|_p \leq b_p \|\Im(T)\|_p. 
\end{equation}
\end{prop} 
Here the constant $b_p$, which will be used below, satisfies the following properties (see \cite{MR0264447} Theorem III.6.3 and its accompanying remark):
\begin{enumerate}
\item[(i)] $p \mapsto b_p$ is monotonically increasing on $[2,\infty)$.
\item[(ii)] If $p\geq 2$ then  $\operatorname{cot}\left(\pi / (2p) \right)  \leq  b_p < p/(\ln(2) e^{2/3})$. 
\item[(iii)]  $b_p =\operatorname{cot} \left(\pi / (2p) \right)$ if $p=2^n, n \in \mn$. In particular, $b_2=1$.
\item[(iv)] If $1<p< 2$ then  $b_{p} = b_{p/(p-1)}$.
\end{enumerate}
In order to state our main result we set
\begin{equation}\label{eq:gamma}
\Gamma_p=\left(1+b_p^{\frac{p}{p-1}} \right)^{p-1}, \quad p>1.
\end{equation}
\begin{rem}\label{rem:estbp}
For later purposes let us note that (i)-(iv) imply that $\Gamma_2=2$ and $\Gamma_p \geq 2$ for all $p >1$.
\end{rem} 
\begin{thm}\label{thm:main}
  Let $A \in \bdd(\hil)$ be selfadjoint and let $B \in \bdd(\hil)$ such that $B-A \in \mathcal{S}_p(\hil)$ for some $p>1$. Then the following holds:
If $p=2$ then
\begin{equation}\label{eq:main1}
  \sum_{\lambda \in \sigma_d(B)} \left( \dist(\Re(\lambda),\sigma(A))^2 + 2|\Im(\lambda)|^2 \right) \leq 2 \|B-A\|_2^2.
\end{equation}
If $p > 2$ then 
\begin{equation}\label{eq:main2}
  \sum_{\lambda \in \sigma_d(B)} \left( \dist(\Re(\lambda),\sigma(A))^p + 2^{p-2} \Gamma_p |\Im(\lambda)|^p \right) 
\leq  4^{p-2} \Gamma_p \|B-A\|_p^p. 
\end{equation}
If $1<p<2$ then 
\begin{equation}\label{eq:main3}
\sum_{\lambda \in \sigma_d(B)}  \left( \dist(\Re(\lambda),\sigma(A))^p + \Gamma_p |\Im(\lambda)|^p \right) 
\leq 12^{2-p} \Gamma_p  \|B-A\|_p^p.
\end{equation}
Here each eigenvalue is counted according to its algebraic multiplicity.
\end{thm}
In the finite-dimensional case, as already remarked above, estimate (\ref{eq:main1}) has been proved by Kahan \cite{MR0369394} and estimates (\ref{eq:main2}) and (\ref{eq:main3}) are consequences of results proved  by Gil' \cite{MR2777268}. The example 
$$ A= \left(
  \begin{array}{cc}
    0 & 1 \\
    1 & 0
  \end{array} \right), \quad B= \left( 
  \begin{array}{cc}
    ib & 1 \\
    0 & ib  
  \end{array}\right), \quad b > 0$$
shows that estimate (\ref{eq:main1}) is sharp.
\begin{cor}\label{cor:1}
  Given the assumptions of Theorem \ref{thm:main} we have
  \begin{equation}
    \sum_{\lambda \in \sigma_d(B)} \dist(\lambda, \sigma(A))^p \leq C_p \|B-A\|_p^p,
  \end{equation}
where $C_2=2$ and
$$C_p = \left\{
    \begin{array}{cl}
      2^{\frac{p}{2}-1}  4^{p-2} \Gamma_p, & p \in (2,\infty), \\[4pt]
      12^{2-p} \Gamma_p, & p \in (1,2).
    \end{array}\right.$$
\end{cor}
\begin{proof}
  Use  Theorem \ref{thm:main}, Remark \ref{rem:estbp} and the estimate 
  \begin{equation*}
(a^q + b^q) \geq \min(2^{1-q},1) (a+b)^{q}, \qquad a,b, q \geq 0.    
  \end{equation*}
\end{proof}

The proof of Theorem \ref{thm:main} will be given in Section \ref{sec:proof}. In the next section, we collect some preparatory results 
mainly concerning block operator matrices.

\section{Some preparatory lemmas}

First, note that given a closed linear subspace $E$ of $\hil$ every operator $T \in \bdd(\hil)$ can be decomposed as
\begin{equation}
  \label{eq:p1}
  T = \left(
    \begin{array}{cc}
      T_1 & T_2 \\
      T_3 & T_4
    \end{array}\right) : E \oplus F \to E \oplus F,
\end{equation}
where $F=E^\perp$ denotes the orthogonal complement of $E$ and so $\hil=E\oplus F$ is the orthogonal sum of $E$ and $F$. Here $T_1 \in \bdd(E)$, $T_2 \in \bdd(F,E)$, $T_3 \in~\bdd(E,F)$ and $T_4 \in \bdd(F)$. More precisely, if $P_E$ and $P_{F}$ denote the orthogonal projections onto $E$ and $F$, respectively, then we have
$$ T_1 = P_E T|_E, \quad T_2 = P_E T|_{F}, \quad T_3 = P_{F} T|_{E} \quad \text{and} \quad T_4=P_{F} T|_{F}.$$
\begin{lemma}
  Let $T,T_i$ be defined as in (\ref{eq:p1}). Then $T$ is selfadjoint if and only if $T_1$ and $T_4$ are selfadjoint and $T_3=T_2^*$.  
\end{lemma}
The next result is due to Bhatia and Kittaneh, see \cite{MR1066826} Theorem 2. 
\begin{lemma}\label{lem:est}
  Let $T,T_i$ be defined as in (\ref{eq:p1}) and let $p \geq 1$. Then $T \in \mathcal{S}_p(\hil)$ if and only if $T_1 \in \mathcal{S}_p(E)$, $T_2 \in \mathcal{S}_p(F,E)$, $T_3 \in \mathcal{S}_p(E,F)$ and $T_4 \in \mathcal{S}_p(F)$. Moreover, in this case we have
  \begin{equation}
    L_p  \|T\|_p^p \leq \sum_{i=1}^4 \|T_i\|_{p}^p \leq  M_p \|T\|_p^p,
  \end{equation}
where 
\begin{equation} \label{eq:const1}
 L_p := \left\{
  \begin{array}{cc}
    2^{2-p}, & p \in [2,\infty) \\
    1, &  p \in [1,2)
  \end{array}\right. \quad \text{and} \quad M_p := \left\{
  \begin{array}{cc}
     1, &  p \in [2,\infty) \\
   2^{2-p}, & p \in [1,2) 
  \end{array}\right. .
\end{equation}
\end{lemma}
Recall that the essential spectrum of $T \in \bdd(\hil)$ is defined as
$$ \sigma_{ess}(T)= \{ \lambda \in \sigma(T) :  \lambda-T \text{ is not a Fredholm operator} \}.$$
\begin{lemma}\label{lem:spec}
  Let $\hil= E \oplus F$ be defined as above and let $S \in \bdd(\hil)$ be given by
$$S=\left(
    \begin{array}{cc}
      S_1 & 0 \\
      0 & S_2
    \end{array}\right) : E \oplus F \to E \oplus F.$$
Then the following holds:
  \begin{enumerate}
  \item[(i)] $\sigma(S)=\sigma(S_1) \cup \sigma(S_2)$.
  \item[(ii)] $\sigma_{ess}(S)=\sigma_{ess}(S_1) \cup \sigma_{ess}(S_2)$.
  \item[(iii)] If $S_1$ and $S_2$ (and so $S$) are selfadjoint, then 
$$\sigma_d(S)= \left[ \sigma_d(S_1) \setminus \sigma_{ess}(S_2) \right]  \cup \left[ \sigma_d(S_2) \setminus \sigma_{ess}(S_1) \right].$$      
  \end{enumerate}
\end{lemma}
\begin{proof}
  The first statement follows from the  fact that $E$ and $F$ are reducing subspaces for $S$, which also implies that  $\lambda-S$ is Fredholm if and only if both $S_1$ and $S_2$ are Fredholm (which
is the second statement). The third statement is a consequence of (i) and (ii)  and the fact that for a selfadjoint operator $T$ we have $\sigma_{d}(T)=\sigma(T) \setminus \sigma_{ess}(T)$. 
\end{proof}
In the proof of the following lemma $\|T\|_\infty$ denotes the operator norm of $T \in \bdd(\hil)$. 
\begin{lemma}\label{lem:inter}
For $K \in \mc^{n \times n}$ let $K_D \in \mc^{n \times n}$ denote its diagonal, i.e.
$$ (K_D)_{ij}= \left\{
  \begin{array}{cl}
    (K)_{ij}, & i=j, \\
    0, & i \neq j, 
  \end{array} \right.$$
and set $K_O=K-K_D$. Then 
\begin{equation}
  \label{eq:ll}
  \|K_D\|_p^p + \|K_O\|_p^p \leq N_p \|K\|_p^p,  
\end{equation}
where
\begin{equation}\label{eq:const2}
N_p= \left\{
  \begin{array}{cl}
    2^{p-2}, & p \in [2,\infty) \\
    3^{2 -p}, & p \in [1,2).
  \end{array}\right.  
\end{equation}
\end{lemma}
We note that the validity of estimate (\ref{eq:ll}) with a constant $\tilde{N}_p=(1+2^p)$ is an immediate consequence of the triangle inequality and the fact that $\|K_D\|_p \leq \|K\|_p$, see \cite{MR1477662}. However, as compared to the constant in (\ref{eq:const2}) this constant does not give the value $1$ for $p=2$, which we need in order to obtain the sharp value $C_2=2$ in (\ref{eq:1}).
\begin{proof}[Proof of Lemma \ref{lem:inter}]
For $p \in [1,\infty]$ define a linear operator $G : \mathcal{S}_p(\mc^n) \to \mathcal{S}_p(\mc^{2n})$ by 
$$ G(K)= \left(
  \begin{array}{cc}
    K_D & 0 \\
    0 & K_O
  \end{array}\right), \quad K \in \mc^{n \times n}.$$
Then
\begin{equation}
  \label{eq:AA}
\|G(K)\|_p = \left( \|K_D\|_p^p + \|K_O\|_p^p\right)^{1/p}, \qquad p \in [1,\infty),  
\end{equation}
and $ \|G(K)\|_\infty= \max( \|K_D\|_\infty, \|K_O\|_\infty)$. Moreover, we have $\|G(K)\|_2=\|K\|_2$ and for $p \in [1,\infty)$ we can estimate 
$$ \|G(K)\|_p^p \leq \|K_D\|_p^p + (\|K\|_p + \|K_D\|_p)^p \leq  (1+2^p)\|K\|_p^p.$$
Denoting the operator norm of $G$ by  $\|G\|_{(p)}$, i.e.
$$ \|G\|_{(p)} = \sup_{K \in \mc^{n \times n}, K \neq 0} \frac{\|G(K)\|_p}{\|K\|_p},$$
we thus obtain that $ \|G\|_{(2)} = 1$, $\|G\|_{(1)} \leq 3$ and $\|G\|_{(\infty)} \leq 2$.
Using interpolation (see, e.g., \cite{MR0358340} Theorem 8) we can conclude that
$$ \|G\|_{(p)} \leq \left\{
  \begin{array}{cl}
    2^{1-\frac 2 p}, & p \in [2,\infty] \\
    3^{\frac 2 p -1}, & p \in [1,2).
  \end{array}\right.$$
This estimate, together with (\ref{eq:AA}), implies the validity of (\ref{eq:ll}).
\end{proof}
\begin{rem}
The constant in estimate (\ref{eq:ll}) is sharp for $p  \in \{1,2\}$ and for $p=\infty$ (when understood in the obvious way). For $p=2$ this was shown in the previous proof, and for $p=1$ and $p=\infty$ it  can be seen by considering the matrix $E \in \mc^{n \times n}$, whose entries are all ones, and the matrix $E-\frac{n}{2} I$, respectively, and sending $n \to \infty$. 
\end{rem}
Our final preparatory result uses one side of the Clarkson-McCarthy inequalities (see \cite{MR0225140}): If $S,T \in \mathcal{S}_p(\hil), p \geq 1,$ then
\begin{equation}
  \label{eq:clark}
  \|T\|_p^p +\|S\|_p^p \leq \frac{M_p}{2} \left( \|S+T\|_p^p + \|S-T\|_p^p \right),
\end{equation}
where $M_p$ was defined in (\ref{eq:const1}).
\begin{lemma}\label{lem:real}
  Let $K \in \mathcal{S}_p(\hil), p \geq 1$. Then
  \begin{equation}
\|\Re(K)\|_p^p + \|\Im(K)\|_p^p \leq M_p \|K\|_p^p.    
  \end{equation}
\end{lemma}
\begin{proof}
Using (\ref{eq:clark}) 
we obtain 
 \begin{eqnarray*}
&&   \|\Re(K)\|_p^p + \|\Im(K)\|_p^p = \|\Re(K)\|_p^p + \|i\Im(K)\|_p^p \\
&\leq& \frac{M_p}{2} \left( \|\Re(K)+ i\Im(K)\|_p^p + \|\Re(K)-i\Im(K)\|_p^p \right) \\
&=& \frac{M_p}{2} \left( \|K\|_p^p + \|K^*\|_p^p \right)= M_p \|K\|_p^p. \
 \end{eqnarray*}
\end{proof}

\section{The proof of Theorem \ref{thm:main}}\label{sec:proof} 
  
Let  $A \in \bdd(\hil)$ be selfadjoint  and let $B \in \bdd(\hil)$ such that $B-A \in \mathcal{S}_p(\hil)$ where $p > 1$.
In the following we fix an arbitrary finite subset $\Lambda \subset \sigma_d(B)$. Let $P_B(\Lambda)$ denote the corresponding Riesz projection (see, e.g., \cite{MR1130394}) and set
$$ E = {\ran(P_B(\Lambda))}, \quad N = \dim(E) \quad \text{and} \quad F=E^\perp.$$
Note that the closed subspace $E$ is the linear span of all eigenvectors and generalized eigenvectors corresponding to the eigenvalues in $\Lambda$ and $N \in \mn$ coincides with the sum of the algebraic multiplicities of these eigenvalues. In particular, $E$ is $B$-invariant and $\sigma(B|_E)=\sigma_d(B|_E)=\Lambda$. For the rest of this proof let us agree that 
\begin{equation}
\lambda_1, \ldots, \lambda_N  
\end{equation}
denote the eigenvalues of $B$ in $\Lambda$, where each eigenvalue is counted according to its algebraic multiplicity. 

With respect to the decomposition $\hil=E \oplus F$ the operator $B$ can be written as (recall that $E$ is $B$-invariant)
$$ B = \left(
  \begin{array}{cc}
    B_1 & B_2 \\
    0 & B_3
  \end{array}\right),$$
where $B_1=B|_E$. Similarly, with respect to the same decomposition we can write
$$ A = \left(
  \begin{array}{cc}
    A_1 & A_2 \\
    A_2^* & A_3
  \end{array}\right),$$ 
where $A_1$ and $A_3$ are selfadjoint. 
\begin{rem}
As a consequence of  Lemma \ref{lem:est} and the fact that the $\mathcal{S}_p$-norm of an operator and its adjoint coincide, we obtain
  \begin{equation}
   \label{eq:5}
 \|B_1-A_1\|_p^p+\|B_2-A_2\|_p^p+\|A_2\|_p^p+\|B_3-A_3\|_p^p   \leq M_p \|B-A\|_p^p,
  \end{equation}
where $M_p$ was defined in (\ref{eq:const1}).  
\end{rem}
Our problem is invariant under unitary similarity, so (invoking Schur's theorem) without loss of generality we can assume that $E=\mc^N$, that $A_1, B_1 \in \mc^{N \times N}$ and that $B_1$ is upper-triangular, i.e.
\begin{equation}
  \label{eq:4}
B_1=\left(
  \begin{array}{ccccc}
    \lambda_{1} & b_{12} & \cdots&  &    b_{1N} \\
    0 & \lambda_{2} & b_{23} & \cdots &   b_{2N} \\
    \vdots  &  \ddots & \ddots & \ddots  &    \\
    \vdots  &   & \ddots &\ddots   & b_{N-1,N} \\
    0  & \cdots & \cdots & 0  & \lambda_{N} 
  \end{array} \right). 
\end{equation}
Next, following the approach of Kahan and Gil',  we will further split up the matrix $B_1$. 
To this end, let us define the hermitian diagonal matrices
$$ R_1= \operatorname{diag}(\Re(\lambda_1), \ldots, \Re(\lambda_N)) \quad \text{and} \quad 
  I_1= \operatorname{diag}(\Im(\lambda_1), \ldots, \Im(\lambda_N)),$$
and  the strictly upper-triangular matrix $U_1=B_1-R_1-iI_1$, i.e.
$$ U_1= 
\left(
  \begin{array}{ccccc}
    0 & b_{12} & \cdots&  &    b_{1N} \\
    0 & 0 & b_{23} & \cdots &   b_{2N} \\
    \vdots  & \ddots & \ddots & \ddots  &    \\
    \vdots  &   & \ddots &\ddots   & b_{N-1,N} \\
    0  & \cdots & \cdots & 0  & 0 
  \end{array} \right). $$
Note that 
  \begin{equation}
    \Re(B_1)= R_1 + \Re(U_1) \quad \text{and} \quad
    \Im(B_1)= I_1 + \Im(U_1).
  \end{equation}
\begin{lemma}
\begin{equation}\label{eq:diag}
\|I_1\|_p^p+ \|\Im(U_1)\|_p^p  \leq N_p \|\Im(B_1)\|_p^p,
   \end{equation}
where $N_p \geq 1$ was defined in (\ref{eq:const2}). 
  \end{lemma}
  \begin{proof}
  Apply Lemma \ref{lem:inter} to  $K=\Im(B_1)=I_1 + \Im(U_1)$. 
  \end{proof}
 \begin{lemma}
We have
  \begin{equation}
      \label{eq:I1}
      \|I_1\|_p^p  = \sum_{k=1}^N |\Im(\lambda_k)|^p
    \end{equation}
and
\begin{equation}
  \label{eq:U1}
  \|\Re(U_1)\|_p \leq b_p \|\Im(U_1)\|_p,
\end{equation}
where $b_p$ was defined in (\ref{eq:bp}).
\end{lemma}
  \begin{proof}
The identity is a direct consequence of the definition of $I_1$ and  the inequality is implied by Proposition \ref{prop:ma} and the fact that $\sigma(U_1)=\{0\}$. 
  \end{proof}
As a final definition let us set
\begin{equation}
  \label{eq:C}
 C = \left(
  \begin{array}{cc}
    R_1 & 0 \\
    0   & A_3
  \end{array}\right) : E \oplus F \to E \oplus F.
\end{equation}
Then $C$ is selfadjoint and 
$$C-A= \left(
  \begin{array}{cc}
    R_1-A_1 & -A_2 \\
    -A_2^* & 0
  \end{array}\right) \in \mathcal{S}_p(\hil).$$
By construction, the points $\Re(\lambda_i), i=1,\ldots,N,$ are eigenvalues of $C$. The next lemma studies when these eigenvalues are isolated. 

\begin{lemma}\label{lem:mult}
Let $\lambda_1, \ldots, \lambda_N$ be as above. Then the following holds:
\begin{enumerate}
\item If $\Re(\lambda_i) \notin \sigma_{ess}(A)$, then $\Re(\lambda_i) \in \sigma_d(C)$.
\item If $\Re(\lambda_i) \in \sigma_d(C)$, then its algebraic multiplicity is not smaller than the algebraic multiplicity of $\lambda_i$ as an eigenvalue of $B$.
\end{enumerate}
\end{lemma}
\begin{proof}
Using Weyl's theorem  and the fact that $\sigma_{ess}(R_1)=\emptyset$ we obtain from Lemma \ref{lem:spec}.(ii) that 
$$\sigma_{ess}(A)= \sigma_{ess}(C)=\sigma_{ess}(R_1) \cup \sigma_{ess}(A_3)=\sigma_{ess}(A_3).$$ 
From Lemma \ref{lem:spec}.(iii) we obtain $ \sigma_d(R_1) \setminus \sigma_{ess}(A_3) \subset \sigma_d(C)$,
so we have 
$$ \sigma_d(R_1) \setminus \sigma_{ess}(A) \subset \sigma_d(C).$$
Since $\Re(\lambda_i) \in \sigma_d(R_1)$ the first statement follows. The second statement is a direct consequence of the definition of $C$ and $\{\lambda_i\}_{i=1}^N$.
\end{proof}
Now we can start with the actual estimate.
\begin{lemma} \label{lem:half}
We have
\begin{equation}
  \label{eq:half}
 \sum_{k=1}^N \dist(\Re(\lambda_k),\sigma(A))^p \leq  \|C-A\|_p^p.  
\end{equation}
\end{lemma} 
\begin{proof} 
Since $C$ and $A$ are selfadjoint we can apply Kato's theorem \cite{MR900507} (i.e. the validity of (\ref{eq:1}) with $C_p=1$) to obtain
$$ \sum_{\mu \in \sigma_d(C)} \dist(\mu,\sigma(A))^p \leq  \|C-A\|_p^p.$$
But Lemma \ref{lem:mult} shows that 
\begin{eqnarray*}
\sum_{k=1}^N \dist(\Re(\lambda_k),\sigma(A))^p &=& \sum_{k \in \{1, \ldots, N \} : {\Re(\lambda_k) \notin \sigma_{ess}(A)}} \dist(\Re(\lambda_k),\sigma(A))^p \\
&\leq&  \sum_{\mu \in \sigma_d(C)} \dist(\mu,\sigma(A))^p. 
\end{eqnarray*}
 \end{proof}
In the following we will provide a suitable upper bound for $\|C-A\|_p^p$.
 \begin{lemma}
We have
  \begin{equation*}
    \|C-A\|_p^p \leq  L_p^{-1} \Gamma_p \left( \|\Re(B_1-A_1)\|_p^p + \|\Im(U_1)\|_p^p + \|A_2\|_p^p\right),
  \end{equation*}
where $L_p$ and $\Gamma_p$ were defined in (\ref{eq:const1}) and (\ref{eq:gamma}), respectively.
 \end{lemma}
 \begin{proof}
From Lemma \ref{lem:est} we obtain
\begin{eqnarray*}
&& \|C-A\|_p^p \leq L_p^{-1}  \left( \|R_1-A_1\|_p^p + 2 \|A_2\|_p^p \right).
\end{eqnarray*}
Recall that  $R_1-A_1=\Re(B_1-A_1)-\Re(U_1)$ . 
So we can use the triangle inequality, estimate (\ref{eq:U1}) and  H\"older's inequality  to obtain that
\begin{eqnarray*}
&& \|R_1-A_1\|_p^p \leq    \left(\|\Re(B_1-A_1)\|_p+ \|\Re(U_1)\|_p \right)^p   \\
&\leq&    \left(\|\Re(B_1-A_1)\|_p+ b_p\|\Im(U_1)\|_p \right)^p   \\
&\leq& (1+b_p^{\frac{p}{p-1}})^{p-1}  \left( \|\Re(B_1-A_1)\|_p^p + \|\Im(U_1)\|_p^p \right).
\end{eqnarray*}
Now recall that $\Gamma_p= (1+b_p^{\frac{p}{p-1}})^{p-1}\geq 2$ (see Remark \ref{rem:estbp}). 
 \end{proof}
The relevance of the next lemma will become clear in a moment.
\begin{lemma}\label{lem:last}
Let $L_p,N_p, M_p$ and $\Gamma_p$ be defined as above. Then
 \begin{equation}
 \|C-A\|_p^p+ L_p^{-1} \Gamma_p \|I_1\|_p^p \leq   L_p^{-1} \Gamma_p N_p M_p^2  \|B-A\|_p^p.
 \end{equation}
\end{lemma}
\begin{proof}
 From the previous lemma and estimate (\ref{eq:diag}) we know that
 \begin{eqnarray*}
&&     \|C-A\|_p^p+  L_p^{-1}  \Gamma_p \|I_1\|_p^p \\
&\leq&       L_p^{-1} \Gamma_p  \left[ \|A_2\|_p^p + \|\Re(B_1-A_1)\|_p^p + \|\Im(U_1)\|_p^p  + \|I_1\|_p^p\right] \\
&\leq&       L_p^{-1} \Gamma_p N_p  \left[ \|A_2\|_p^p + \|\Re(B_1-A_1)\|_p^p + \|\Im(B_1)\|_p^p\right],
 \end{eqnarray*}
where $N_p \geq 1$ was defined in (\ref{eq:const2}). Next, apply Lemma \ref{lem:real} to obtain 
 \begin{eqnarray*}
 \|C-A\|_p^p+ L_p^{-1}\Gamma_p\|I_1\|_p^p \leq   L_p^{-1} \Gamma_p N_p M_p  \left[ \|A_2\|_p^p + \|B_1-A_1\|_p^p\right],
 \end{eqnarray*}
where $M_p \geq 1$ was defined in (\ref{eq:const1}). Finally, an application of  (\ref{eq:5}) leads to the desired result.
\end{proof}

Now we can finish the proof of Theorem \ref{thm:main}: Using (\ref{eq:I1}) we obtain from  Lemma \ref{lem:half}  and Lemma \ref{lem:last} that 
\begin{equation*}
\sum_{k=1}^N \left( \dist(\Re(\lambda_k),\sigma(A))^p +  L_p^{-1} \Gamma_p |\Im(\lambda_k)|^p \right)  \leq   L_p^{-1} \Gamma_p N_p M_p^2 \|B-A\|_p^p.
\end{equation*}
Since $\Lambda=\{\lambda_1, \ldots, \lambda_N\}$ was an arbitrary finite subset of $\sigma_d(B)$ and the right-hand side of the last inequality is independent of $\Lambda$, we can conclude that 
\begin{eqnarray*}
\sum_{\small{\lambda \in \sigma_d(B)}} \left( \dist(\Re(\lambda),\sigma(A))^p + L_p^{-1} \Gamma_p |\Im(\lambda)|^p \right)  \leq L_p^{-1} \Gamma_p N_p M_p^2 \|B-A\|_p^p.
\end{eqnarray*}
All that remains is to evaluate the constants.

\section{An Application}

 We start this section with a version of Corollary \ref{cor:1} for perturbations of unitary operators. Recall that the spectrum of a unitary operator is a subset of the unit circle $\mt= \partial \md$. 
\begin{thm}\label{thm:ap1}
  Let $U \in \bdd(\hil)$ be unitary with $\sigma(U) \neq \mt$  and let $V \in \bdd(\hil)$ such that $V-U \in \mathcal{S}_p(\hil)$ for some $p>1$.
Moreover, let $a \in \mt \setminus (\sigma(U) \cup \sigma(V))$. Then
\begin{equation}
 \sum_{\lambda \in \sigma_d(V)} \frac{\dist(\lambda, \sigma(U))^p}{|a-\lambda|^p} \leq  C_p 2^p\|(a-V)^{-1}- (a-U)^{-1}\|_p^p,
\end{equation}
where $C_p$ was defined in Corollary \ref{cor:1}.
\end{thm}
\begin{rem}\label{rem:fin}
Note that $\sigma(U) \neq \mt$ if and only if $\sigma_{ess}(U) \neq \mt$. Moreover, by Weyl's theorem we have $\sigma_{ess}(V)=\sigma_{ess}(U) \subsetneq \mt$ and so $\mc \setminus \sigma_{ess}(V)$ is connected. This implies that the spectrum of $V$ in $\mc \setminus \sigma_{ess}(V)$ is discrete, see \cite{b_Davies} Theorem 4.3.18. In particular, $\mt \setminus ( \sigma(U) \cup \sigma(V))$ is non-empty whenever $V-U$ is compact and $\sigma(U) \neq \mt$. 
\end{rem} 
The above theorem complements (and in many cases improves) a result of Golinskii and Favorov, see \cite{MR2500508} Theorem 4. See also \cite{FG12}.

\begin{proof}[Proof of Theorem \ref{thm:ap1}]
We define a conformal map $ \phi : \mc \setminus \{a\} \to \mc$ as
$$  \phi(\lambda)= i \frac{a+\lambda}{a-\lambda},$$ 
so $\phi(\md)= \{ \mu: \Re(\mu)> 0 \}$ and $\phi(\mt \setminus \{a\})=\mr$.
Furthermore, let us define the inverse Cayley transforms of $U$ and $V$ as
$$A=\phi(U)=i (a-U)^{-1}(a+U), \qquad B=\phi(V)=i(a-V)^{-1} (a+V).$$ 
Note that $A$ is selfadjoint and by spectral mapping we have  $\sigma(A)=\phi(\sigma(U))$. The spectral mapping theorem also implies that $\lambda \in \sigma_d(V)$ if and only if $\phi(\lambda) \in \sigma_d(B)$, the algebraic multiplicities being preserved. Finally, a short calculation shows that 
\[  A-B  = 2ai \left[ (a-U)^{-1}-(a-V)^{-1} \right] = 2ai(a-U)^{-1}(U-V)(a-V)^{-1},\]
so $A-B \in \mathcal{S}_p(\hil)$ and we can apply Corollary \ref{cor:1} to obtain that  
\begin{eqnarray*}
 \sum_{\lambda \in \sigma_d(V)} \dist(\phi(\lambda), \phi(\sigma(U)) )^p \leq C_p 2^p \| (a-V)^{-1}-(a-U)^{-1} \|_p^p.
\end{eqnarray*}
It remains to note that 
\begin{eqnarray*}
  \dist(\phi(\lambda), \phi(\sigma(U))) &=& \inf_{\xi \in \sigma(U)} \left|\frac{a+\lambda}{a-\lambda} - \frac{a+\xi}{a-\xi}\right| 
= \frac{2}{|a-\lambda|} \inf_{\xi \in \sigma(U)} \left|\frac{\lambda-\xi}{a-\xi} \right| \\
&\geq & \frac{1}{|a-\lambda|} \inf_{\xi \in \sigma(U)} \left|\lambda-\xi \right| 
=  \frac{1}{|a-\lambda|} \dist(\lambda, \sigma(U)) . 
\end{eqnarray*}
\end{proof}
   
In the following we will apply the previous theorem to obtain new results about the distribution of zeros of a class of holomorphic functions on the unit disk, namely, the class $\mathcal{K}$ of all Cauchy transforms of complex Borel measures on the unit circle. It consists of all holomorphic functions $h$ of the form 
\begin{equation}
  \label{eq:K}
 h(w) = \int_\mt \frac{\mu(d\zeta)}{1-\overline{\zeta}w}, \qquad w \in \md,  
\end{equation}
where $\mu$ is some finite, complex Borel measure on $\mt$. We recall that $\mathcal{K}$ contains the classical Hardy spaces $H^q(\md), q \geq 1$. More precisely, we have
\begin{equation*}
 \bigcup_{q \geq 1} H^q(\md) \subsetneq \mathcal{K} \subsetneq \bigcap_{ 0 < q < 1} H^q(\md).  
\end{equation*}
A proof of the above inclusions and many additional results about Cauchy transforms can be found in the monograph \cite{Cauchy}.

What can be said about the distribution of zeros of a Cauchy transform $h \in \mathcal{K}$? Assuming that $h$ is not identically zero, the classical answer is that its zero-set has to satisfy the so-called  \emph{Blaschke condition}, i.e. 
\begin{equation}\label{eq:Blaschke}
  \sum_{h(w)=0, \: w\in \md} (1-|w|) < \infty,
\end{equation}
where each zero is counted according to its order. Indeed, every function in the Hardy class $H^q(\md), q> 0,$ has to satisfy this condition and so does every Cauchy transform. However, it turns out that one can actually say more about the zero-set. To this end, let us first note that in case $\supp(\mu) \neq \mt$ the function $h$ defined in (\ref{eq:K}) can be analytically extended to $\mc \setminus \supp(\mu)$, the complement of the topological support of $\mu$. In particular, the zeros of $h$ can accumulate at $\supp(\mu)$ only, so it seems natural to conjecture that the Blaschke condition (\ref{eq:Blaschke}) can be replaced with the condition that
\begin{equation}
  \label{eq:Blaschke2}
 \sum_{h(w)=0, \: w \in \md} \dist(w, \supp(\mu)) < \infty.
\end{equation}
While we can neither prove nor disprove this conjecture, we can prove a weaker version of (\ref{eq:Blaschke2}).
\begin{thm}\label{thm:ap2}
 Let $\mu$ be a finite, complex Borel measure on $\mt$ with $\supp(\mu) \neq \mt$ and $\mu(\mt) \neq 0$. Moreover, let
$$g(w) = \int_\mt \frac{\mu(d\zeta)}{1-\overline{\zeta}w}, \qquad w \in \mc \setminus \supp(\mu).$$
Then for every $p>1$ we have 
\begin{equation}
    \label{eq:Blaschke3}
 \sum_{g(w)=0} \dist(w, \supp(\mu))^p < \infty,
\end{equation}
where the sum is over all zeros of $g$ in $\mc \setminus \supp(\mu)$ and each zero is counted according to its order.
\end{thm}

\begin{rem}  
(i) Estimate (\ref{eq:Blaschke3}) seems to be new. We will prove it using Theorem \ref{thm:ap1}, i.e. via operator theory. We don't know how (or whether) it can be proven via a classical complex-analysis argument as well. 

\noindent (ii) If it could be shown that (\ref{eq:Blaschke3}) does not necessarily hold for $p=1$, then the same would be true of Theorem \ref{thm:main}. So this opens a possibility to tackle that problem.
 
\noindent (iii) The idea to use operator theoretic arguments to prove results about zeros of Cauchy transforms has been used before, see \cite{HK13}.

\end{rem}
  
\begin{proof}[Proof of Theorem \ref{thm:ap2}] 
It is no restriction to assume that $\mu(\mt)=g(0)=1$. Denoting the total variation measure of $\mu$ by $|\mu|$, we have $d\mu = \nu d|\mu|$ for some measurable function $\nu : \mt \to \mt$. We are going to apply Theorem \ref{thm:ap1} to certain operators on the Hilbert space $\hil= L^2(\mt,d|\mu|)$. That is, we first define a unitary operator $U$ on $\hil$ by setting 
$$(Uf)(\zeta) = \overline{\zeta} f(\zeta).$$   
Note that $\sigma(U)= \{ \zeta \in \mt : \overline{\zeta} \in \supp(\mu)\}$ and $\sigma(U^*)=\supp(\mu)$. Next, we define a rank one operator $L$ on $\hil$ as $Lf = -\langle f, \psi \rangle \phi$, where
$$\phi(\zeta)={\overline{\zeta}}, \qquad \psi(\zeta)= \overline{\nu(\zeta)}.$$ 
Finally, we set $V=U+L$. Note that the spectrum of $V$ in $\mc \setminus \sigma(U)$ is discrete. Moreover,
some $\lambda_0 \in \mc \setminus \sigma(U)$ is in $\sigma_d(V)$ if and only if $\lambda_0$ is a zero of the analytic function
$$ d: \mc \setminus \sigma(U) \ni \lambda \mapsto \det(I-L(\lambda-U)^{-1})$$  
and the multiplicity of $\lambda_0$ as an eigenvalue of $V$ coincides with its order as a zero of $d$, see e.g. \cite{b_Gohberg69}, p.173-174. 
Setting $\lambda=w^{-1}$ and noting that $w \in \mc \setminus \supp(\mu)$ iff $w^{-1} \in \mc \setminus \sigma(U)$, we then compute (recall that $\mu(\mt)=1$)
\begin{eqnarray*}
d(1/w) &=&\det(I-wL(I-wU)^{-1}) = 1+ w \langle (I-wU)^{-1} \phi, \psi \rangle  \\
&=& 1 + \int_{\mt} \frac {w \overline{\zeta}} {1-w\overline{\zeta}}  \mu(d\zeta) 
= \int_{\mt} \frac {1} {1-w\overline{\zeta}}  \mu(d\zeta) = g(w).
\end{eqnarray*}
So we see that $w \in \mc \setminus \supp(\mu)$ is a zero of $g$ if and only if $w^{-1} \in \sigma_d(V)$. 

Since the zero-set of $g$ is discrete and since we assumed that $\supp(\mu) \neq \mt$, there exists $a \in \mt \setminus \supp(\mu)$ (i.e. $\overline{a} \in \mt \setminus \sigma(U)$) with $g(a) \neq 0$. The previous equivalence then shows that $\overline{a}=a^{-1} \in \mt \setminus ( \sigma(U) \cup \sigma(V))$, so we can use Theorem \ref{thm:ap1} to obtain that 
\begin{equation}
  \label{eq:f1}
\sum_{g(w)=0} \frac{\dist(w^{-1}, \sigma(U))^p}{|w^{-1}-\overline{a}|^p}  \leq \sum_{\lambda \in \sigma_d(V)} \frac{\dist(\lambda, \sigma(U))^p}{|\lambda-\overline{a}|^p} < \infty, \qquad p > 1.
\end{equation} 
Since
\begin{eqnarray*}
\frac{\dist(w^{-1}, \sigma(U))}{|w^{-1}-\overline{a}|} &=&   \frac{\dist(w, \sigma(U^*))}{|a-w|} =  \frac{\dist(w, \supp(\mu))}{|a-w|},
\end{eqnarray*}
we arrive at
\begin{equation}
  \label{eq:33}
 \sum_{g(w)=0} \frac{\dist(w, \supp(\mu))^p}{|a-w|^p} < \infty.  
\end{equation}
But the zeros of $g$ cannot accumulate at infinity, so (\ref{eq:33}) implies (\ref{eq:Blaschke3}).  
\end{proof} 

\section*{Acknowledgments}
  
I would like to thank D. Wenzel and G. Katriel for some helpful discussions. 

\bibliography{bibliography}
\bibliographystyle{plain}

\end{document}